\documentclass[10pt,twoside]{article}
\usepackage{natbib}
\usepackage{amsmath,amsthm,amsopn,amstext,amscd,amsfonts,amssymb,color}

\def\diag{\mathop{\rm diag}\nolimits}

\def\build#1#2#3{\mathrel{\mathop{#1}\limits^{#2}_{#3}}}

\def\Vec{\mathop{\rm vec}\nolimits}
\def\Cov{\mathop{\rm Cov}\nolimits}

\def\E{\mathop{\rm E}\nolimits}

\renewenvironment{abstract}
                 {\vspace{6pt}
                  \begin{center}
                  \begin{minipage}{5.4in}
                  \rule{5.4in}{.1mm}\\[2ex]
                  \small
                  \textbf{Abstract}\\[2ex]
                  \hspace{.5cm}}
                 {\end{minipage}\end{center}}
\newcommand{\keywords}[1]
           {\begin{center}
            \small
            \begin{minipage}{5.4in}
            \noindent \textit{Key words:}~{\textrm{#1}}
            \\ \rule{5.4in}{.1mm}\\[2ex]
            \end{minipage}
            \end{center}
            \normalsize
           }

\newcommand{\AMS}[1]
           {\begin{center}
            \small
            \begin{minipage}{5.4in}
            \noindent \textit{2000 Mathematical Subject Classification:}~{primary \textrm{#1}.}
            \end{minipage}
            \end{center}
            \normalsize
           }

\newtheorem{theorem}{Theorem}[section]

\newtheorem{corollary}[theorem]{Corollary}

\title{\textbf{Asymptotic normality of the optimal solution in multiresponse surface methodology}}
\author{\textbf{Jos\'e A. D\'{\i}az-Garc\'{\i}a},\\
  \textit{\small Department of Statistics and Computation,}\\
  \textit{\small  Universidad Aut\'onoma Agraria Antonio Narro},\\
  \textit{\small 25350 Buenavista, Saltillo, Coahuila, Mexico}.\\
  \textit{\small E-mail: jadiaz@uaaan.mx}\\
  \textbf{Francisco J. Caro-Lopera},\\
  \textit{\small Departmento de Ciencias B\'{a}sicas},\\
  \textit{\small Universidad de Medell\'{\i}n},\\
  \textit{\small Carrera 87 No. 30 - 65, Medell\'{\i}n, Colombia},\\
  \textit{\small 36240 Guanajuato, M\'exico}.\\
  \textit{\small E-mail: fjcaro@udem.edu.co}\\
  }
\date{}
\begin{document}
\maketitle
\begin{abstract}
In this work is obtained an explicit form for the perturbation effect on the matrix of
regression coefficients on the optimal solution in multiresponse surface methodology. Then, the
sensitivity analysis of the optimal solution is studied and the critical point characterisation
of the convex program, associated with the optimum of a multiresponse surface, is also
analysed. Finally, the asymptotic normality of the optimal solution is derived by standard
methods.
\end{abstract}
\AMS{62K20, 90C25, 90C31}

\section{\normalsize Introduction}\label{Sec1}

The multiresponse surface methodology has been considered as  a very useful tool in the study of
designs, phenomena and experiments. Which enables us to propose a set of analytical
relationship between responses and controlled variables through a process of continuous
improvement and optimisation.

It is assumed that a researcher knows a system and a corresponding  set of observable responses
variables $Y_{1}, \cdots, Y_{r}$ which depends on some input variables, $x_{1}, \dots x_{n}$.
This also supposes that that the input  variables $x_{i}^{'s}$ can be controlled by the
researcher with a minimum error.

Typically we have that
\begin{equation}\label{rrs}
    Y_{k}(\mathbf{x}) = \eta_{k}(x_{1}, \dots x_{n}), \quad k = 1, \dots,r, \mbox{ and }
    \mathbf{x} =(x_{1}, \dots x_{n})',
\end{equation}
where the form of the functions $\eta_{k}(\cdot)$'s are unknown and perhaps, very complex, and it is
usually termed as the true response surface. The success of the response surfaces methodology
depends on the approximation of $\eta_{k}(\cdot)$ for a polynomial of low degree in some
region.

For purposes of this paper is assumed that $\eta_{k}(\cdot)$ can be soundly approximated by a
polynomial of second order, that is
\begin{equation}\label{ersp}
   Y_{k}(\mathbf{x}) = \beta_{0k} + \displaystyle\sum_{i=1}^{n}\beta_{ik}x_{i} + \sum_{i=1}^{n}
    \beta_{iik}x_{i}^{2} +  \sum_{i=1}^{n}\sum_{j>i}^{n}\beta_{ijk}x_{i}x_{j}
\end{equation}
where the unknown parameters $\beta_{j}^{'s}$ can be estimated via regression's techniques, as
it will be described in next section.

Next, we are interested in obtaining the levels of the input variables $x_{i}^{'s}$ such that
the response variables $Y_{1}, \cdots, Y_{r}$ are simultaneously minimal (optimal). This can be
achieved if the following multiobjetive mathematical program is solved
\begin{equation}\label{opp}
  \begin{array}{c}
    \build{\min}{}{\mathbf{x}}
    \left(%
    \begin{array}{c}
      Y_{1} (\mathbf{x}) \\
      Y_{2} (\mathbf{x}) \\
      \vdots\\
      Y_{r} (\mathbf{x}) \\
    \end{array}%
    \right)\\
    \mbox{subject to}\\
    \mathbf{x} \in \mathfrak{X},
  \end{array}
\end{equation}
where $\mathfrak{X}$ is certain operating region for the input variables $x_{i}^{'s}$.

Now, two questions, closely related, can be observed:
\begin{enumerate}
  \item When the estimations of (\ref{ersp}) for $k = 1, \dots, r$ are considered into (\ref{opp}), the critical
  point $\mathbf{x}^{*}$ obtained as solution shall be a function of the estimators
  $\widehat{\beta}_{j}^{\ 's}$ of the $\beta_{j}^{'s}$. Thus, given that $\widehat{\beta}_{j}^{\ 's}$ are random variables,
  then $\mathbf{x}^{*}\equiv \mathbf{x}^{*}(\widehat{\beta}_{j}^{\ 's})$ is a random vector
  too. So, under the assumption that the distribution of $\widehat{\boldsymbol{\beta}}$ is known, then, what is the
  distribution of $\mathbf{x}^{*}(\widehat{\beta}_{j}^{\ 's})$?
  \item And, perhaps it is not sufficient to know only a point estimate of $\mathbf{x}^{*}(\widehat{\beta}_{j}^{\ 's})$,
 could be more convenient to know a estimated region or a estimated interval.
\end{enumerate}
In particular, the distribution of the critical point in a univariate response surface model
was studied by \citet{dr:01,dre:02}, when $y(\mathbf{x})$ is defined as an hyperplane.

Now, in the context of the mathematical programming problems, the sensitivity analysis studies
the effect of small perturbations in: (1) the parameters on the optimal objective function
value and (2) the critical point. In general, these parameters shape the objective function and
constraint the approach to the mathematical programming problem. In particular, \citet{j:77},
\citet{d;84} and \citet{fg:82} have studied the sensitivity analysis of the mathematical
programming, among many other authors. As an immediate consequence of the sensitivity analysis
emerges the asymptotic normality study of the critical point, which can be performed by
standard methods of mathematical statistics (see similar results for the case of maximum
likelihood estimates \citet{as:58}). This last consequence makes the sensitivity analysis very
appealing for researching from a statistical point of view. However, this approach must be
fitted into the classical philosophy of the sensitivity analysis; i.e., we need to translate
into the statistical language, the general sensitivity analysis methodology, which deals with a
number of ways in which the estimators of certain model are affected by omission of a
particular set of variables or by the inclusion or omission of a particular observation or set
of observations, see \citet{chh:88}.

This papers pursues to important aims:  the effect of perturbations of the matrix of regression
parameters on the optimal solution of the multiresponse surface model and  the asymptotic
normality of the critical point. First, in Section \ref{sec2} some notation is defined. Then,
the multiresponse surface mathematical program is proposed in Section \ref{sec3} as a
multiobjective mathematical programming problem and a general solution is considered in terms
of a functional. The characterisation of the critical point is given in Section \ref{sec4} by
stating the first-order and second-order Kuhn-Tucker conditions. Finally, the asymptotic
normality of a critical point is established in Section \ref{sec5} and for a particular form of
the functional, the asymptotic normality of a critical point is also derived.

\section{\normalsize Notation}\label{sec2}

For convenience,  the principal properties and usual notations are given here. A detailed
discussion of the multiresponse surface methodology can be found in \citet[Chap. 7]{kc:87} and
\citet{kc:81}.

Let $N$ be the number of experimental runs and $r$ be the number of response variables, which
can be measured for each setting of a group  of $n$ coded variables $x_{1},  x_{2}, \dots,
x_{n}$. It is assumed that the response variables can be modelled by a second order polynomial
regression model in terms of $x_{i}$, $i = 1,\dots,n$. Hence,  the $k^{th}$ response model can
be written as
\begin{equation}\label{lmu}
    \mathbf{Y}_{k} = \mathbf{X}_{k}\boldsymbol{\beta}_{k} + \boldsymbol{\varepsilon}_{k}, \quad
    k = 1, \dots, r,
\end{equation}
where $\mathbf{Y}_{k}$ is an $N \times 1$ vector of observations on the $k^{th}$ response,
$\mathbf{X}_{k}$ is an $N \times p$ matrix of rank $p$ termed the design or regression matrix,
$p = 1 + n + n(n+1)/2$,  $\boldsymbol{\beta}_{k}$ is a $p \times 1$ vector of unknown constant
parameters,  and $\boldsymbol{\varepsilon}_{k}$ is a random error vector associated with the
$k^{th}$ response. For purposes of this study is assumed that $\mathbf{X}_{1} = \cdots =
\mathbf{X}_{r} = \mathbf{X}$. Therefore, (\ref{lmu}) can be written as
\begin{equation}\label{lmm}
    \mathbf{Y} = \mathbf{X}\mathbb{B} + \mathbb{E}
\end{equation}
where $\mathbf{Y} = \left[\mathbf{Y}_{1}\vdots \mathbf{Y}_{2} \vdots \cdots \vdots
\mathbf{Y}_{r}\right]$,  $\mathbb{B} = \left[\boldsymbol{\beta}_{1}\vdots
\boldsymbol{\beta}_{2} \vdots \cdots \vdots \boldsymbol{\beta}_{r}\right]$, moreover
$$
  \boldsymbol{\beta}_{k} = (\beta_{0k},\beta_{1k},\dots,\beta_{nk},\beta_{11k},\dots,\beta_{nnk},\beta_{12k},
  \dots,\beta_{(n-1)nk})'
$$
and $\mathbb{E} = \left[\boldsymbol{\varepsilon}_{1}\vdots \boldsymbol{\varepsilon}_{2} \vdots
\cdots \vdots \boldsymbol{\varepsilon}_{r}\right]$,  such that $\mathbb{E} \sim \mathcal{N}_{N
\times r} (\mathbf{0},  \mathbf{I}_{N} \otimes \mathbf{\Sigma})$ i.e. $\mathbb{E}$ has an $N
\times r$ matrix multivariate normal distribution with $\E (\mathbb{E}) = \mathbf{0}$ and
$\Cov(\Vec \mathbb{E}') = \mathbf{I}_{N} \otimes \mathbf{\Sigma}$, where $\mathbf{\Sigma}$ is
a $r \times r$ positive definite matrix. Now, if  $\mathbf{A} = \left[\mathbf{A}_{1}\vdots
\mathbf{A}_{2}\vdots \cdots \vdots \mathbf{A}_{r}\right]$, with $\mathbf{A}_{j}$, $j =
1,\cdots, r$  the columns of $\mathbf{A}$; then $\Vec \mathbf{A} = (\mathbf{A}'_{1},
\mathbf{A}'_{2}, \dots, \mathbf{A}'_{r})'$ and $\otimes$ denotes the direct (or Kronecker)
product of matrices, see \citet[Theorem 3.2.2,  p. 79]{mh:82}. In addition denote
\begin{description}
\item[-] $\mathbf{x} = \left(x_{1},  x_{2}, \dots, x_{n}\right)'$: The vector of controllable
variables or factors. Formally,  an $x_{i}$ variable is associated with each factor $A,  B,
...$

\item[-] $\widehat{\mathbb{B}} = \left[\widehat{\boldsymbol{\beta}}_{1}\vdots
\widehat{\boldsymbol{\beta}}_{2} \vdots \cdots \vdots \widehat{\boldsymbol{\beta}}_{r}\right]$:
The least squares estimator of $\mathbb{B}$ given by $\widehat{\mathbb{B}} =
(\mathbf{X}'\mathbf{X})^{-1}\mathbf{X}'\mathbf{Y}$,  from where
$$
  \widehat{\boldsymbol{\beta}}_{k} = (\mathbf{X}' \mathbf{X})^{-1} \mathbf{X}' \mathbf{Y}_{k} =
  (\widehat{\beta}_{0k},\widehat{\beta}_{1k},\dots,\widehat{\beta}_{n},\widehat{\beta}_{11k},
  \dots,\widehat{\beta}_{nnk},\widehat{\beta}_{12k},\dots,\widehat{\beta}_{(n-1)nk})'
$$
$k = 1,  2,  \dots,  r$. Moreover,  under the assumption that $\mathbb{E} \sim
\mathcal{N}_{N \times r} (\mathbf{0},  \mathbf{I}_{N} \otimes \mathbf{\Sigma})$,  then
$\widehat{\mathbb{B}} \sim \mathcal{N}_{p \times r} (\mathbb{B},  (\mathbf{X}'\mathbf{X})^{-1}
\otimes \mathbf{\Sigma})$,  with $\Cov(\Vec \widehat{\mathbb{B}}') =
(\mathbf{X}'\mathbf{X})^{-1} \otimes \mathbf{\Sigma}$.

\item[-] $\mathbf{z}(\mathbf{x}) = (1, x _{1},  x_{2},  \dots,  x_{n},  x _{1}^{2},  x_{2}^{2},
\dots,  x_{n}^{2},  x _{1} x_{2},  x_{1}x_{3}\dots,  x_{n-1}x_{n})'.$

\item[-] $\widehat{\boldsymbol{\beta}}_{1k}=
(\widehat{\beta}_{1k},\dots,\widehat{\beta}_{nk})'$ and
$$
  \widehat{\mathbf{B}}_{k} = \frac{1}{2}
  \left (
     \begin{array}{cccc}
       2\widehat{\beta}_{11k} & \widehat{\beta}_{12k} & \cdots & \widehat{\beta}_{1nk} \\
       \widehat{\beta}_{21k} & 2\widehat{\beta}_{22k} & \cdots & \widehat{\beta}_{2nk} \\
       \vdots & \vdots & \ddots & \vdots \\
       \widehat{\beta}_{n1k} & \widehat{\beta}_{n2k} & \cdots & 2\widehat{\beta}_{nnk}
     \end{array}
  \right)
$$

\item[$\centerdot$] \hspace{-3cm} \vspace{-.85cm}
\begin{eqnarray*}
  \widehat{Y}_{k} (\mathbf{x}) &=& \mathbf{z}'(\mathbf{x})\widehat{\boldsymbol{\beta}}_{k} \\
&=& \widehat{\beta}_{0k} + \displaystyle\sum_{i=1}^{n}\widehat{\beta}_{ik}x_{i} +
\sum_{i=1}^{n}
    \widehat{\beta}_{iik}x_{i}^{2} +
    \sum_{i=1}^{n}\sum_{j>i}^{n}\widehat{\beta}_{ijk}x_{i}x_{j}\\
    &=& \widehat{\beta}_{0k}+ \widehat{\boldsymbol{\beta}}'_{1k}\mathbf{x}+\mathbf{x}^{'}\widehat{\mathbf{B}}_{k}\mathbf{x}:
\end{eqnarray*}
The response surface or predictor equation at the point $\mathbf{x}$ for the k$^{th}$ response
variable.

\item[-] $\widehat{\mathbf{Y}} (\mathbf{x}) = \left(\widehat{Y}_{1} (\mathbf{x}),  \widehat{Y}_{2}
(\mathbf{x}),  \dots,  \widehat{Y}_{r} (\mathbf{x})\right)' = \widehat{\mathbb{B}}'
\mathbf{z}(\mathbf{x})$: The multiresponse surface or predicted response vector at the point
$\mathbf{x}$.

\item[-]$\widehat{\mathbf{\Sigma}} = \displaystyle \frac{\mathbf{Y}'(\mathbf{I}_{N} -
\mathbf{X}(\mathbf{X}'\mathbf{X})^{-1}\mathbf{X}')\mathbf{Y}}{N-p}$: The estimator of the
variance-covariance matrix $\mathbf{\Sigma}$ such that $(N-p)\widehat{\mathbf{\Sigma}}$ has a
Wishart distribution with $(N-p)$ degrees of freedom and the parameter $\mathbf{\Sigma}$; this
fact is denoted as $(N-p)\widehat{\mathbf{\Sigma}} \sim \mathcal{W}_{r}(N-p, \mathbf{\Sigma})$.
Here,  $\mathbf{I}_{m}$ denotes an identity matrix of order $m$.

\item[-] Finally,  note that
\begin{equation}\label{eq5}
    E(\widehat{\mathbf{Y}}(\mathbf{x})) =
    E(\widehat{\mathbb{B}}'\mathbf{z}(\mathbf{x})) = \mathbb{B}'\mathbf{z}(\mathbf{x})
\end{equation}
and
\begin{equation}\label{eq6}
    \Cov (\widehat{\mathbf{Y}}(\mathbf{x})) = \mathbf{z}'(\mathbf{x})(\mathbf{X}'
    \mathbf{X})^{-1}\mathbf{z}(\mathbf{x}) \mathbf{\Sigma}.
\end{equation}
An unbiased estimator of $\Cov (\widehat{\mathbf{Y}}(\mathbf{x}))$ is given by
\begin{equation}\label{eq7}
   \widehat{ \Cov} (\widehat{\mathbf{Y}}(\mathbf{x})) = \mathbf{z}'(\mathbf{x})(\mathbf{X}'
    \mathbf{X})^{-1}\mathbf{z}(\mathbf{x}) \widehat{\mathbf{\Sigma}}.
\end{equation}
\end{description}

\section{\normalsize Multiresponse surface mathematical programming}\label{sec3}

In the following sections,  we make use of the multiresponse mathematical programming and
multiobjective mathematical programming. For convenience,  the concepts and notations required
are listed below in terms of the estimated model of multiresponse surface mathematical
programming. Definitions and detailed properties can be found in \citet{kc:81},  \citet{kc:87},
\citet{rio:89}, \citet{s86}, and \citet{m99}.

The multiresponse mathematical programming or multiresponse optimisation (MRO) problem is proposed, in
general, as follows
\begin{equation}\label{equiv1}
  \begin{array}{c}
    \build{\min}{}{\mathbf{x}}\widehat{\mathbf{Y}}(\mathbf{x}) = \
    \build{\min}{}{\mathbf{x}}
    \left(%
    \begin{array}{c}
      \widehat{Y}_{1} (\mathbf{x}) \\
      \widehat{Y}_{2} (\mathbf{x}) \\
      \vdots\\
      \widehat{Y}_{r} (\mathbf{x}) \\
    \end{array}%
    \right)\\
    \mbox{subject to}\\
    \mathbf{x} \in \mathfrak{X}.
  \end{array}
\end{equation}
It is a nonlinear multiobjective mathematical  programming problem,  see \citet{s86},
\citet{rio:89} and \citet{m99}; and  $\mathfrak{X}$ denotes the experimental region, usually
taken as a hypercube
$$
  \mathfrak{X}= \{\mathbf{x}|l_{i} < x_{i} < u_{i},  \quad i = 1,  2,  \dots,  n\},
$$
where $\mathbf{l} = \left(l_{1},  l_{2},  \dots,  l_{n}\right)'$,  defines the vector of lower
bounds of factors and $\mathbf{u} = \left(u_{1},  u_{2},  \dots,  u_{n}\right)'$,  gives the
vector of upper bounds of factors. Alternatively,  the experimental region can taken as a
hypersphere
$$
  \mathfrak{X}= \{\mathbf{x}|\mathbf{x}'\mathbf{x} \leq c^{2},   c \in \Re\},
$$
where,   $c$ is set according to the experimental design model under consideration,  see \citet{kc:87}.
Alternatively (\ref{equiv1}) can be written as
$$
  \build{\min}{}{\mathbf{x} \in \mathfrak{X}}\widehat{\mathbf{Y}}(\mathbf{x}).
$$

In the response surface methodology context,  the multiobjective mathematical programs  rarely
contains a point $\mathbf{x^{*}}$ which can be considered as an optimum, i.e. few cases satisfy
the requirement that $ \widehat{Y}_{k} (\mathbf{x}) $ is minimum for all $k = 1, 2, \dots,  r$.
From the viewpoint of multiobjective mathematical programming,  this justifies the following
notion of the \emph{Pareto point}:
\begin{center}
\begin{minipage}[t]{4.5in}
\begin{it}
We say that $\widehat{\mathbf{Y}}^{*} (\mathbf{x})$ is a  \textit{Pareto point} of
$\widehat{\mathbf{Y}} (\mathbf{x})$,  if there is no other point $\widehat{\mathbf{Y}}^{1}
(\mathbf{x})$ such that $\widehat{\mathbf{Y}}^{1} (\mathbf{x}) \leq \widehat{\mathbf{Y}}^{*}
(\mathbf{x})$,  i.e. for all $k$,  $\widehat{Y}_{k}^{1} (\mathbf{x}) \leq \widehat{Y}_{k}^{*}
(\mathbf{x})$ and $\widehat{\mathbf{Y}}^{1} (\mathbf{x}) \neq \widehat{\mathbf{Y}}^{*}
(\mathbf{x})$.
\end{it}
\end{minipage}
\end{center}
\citet{s86},  \citet{rio:89} and \citet{m99} established the existence criteria for Pareto points
in a multiobjective mathematical programming problem  and
the extension of scalar mathematical programming (\textit{Kuhn-Tucker's conditions}) to the
vectorial case.

Methods for solving a multiobjective mathematical program are based on the existing information
about a particular problem. There are three possible scenarios: when the investigator
possesses either complete,  partial or null information,  see \citet{rio:89},  \citet{m99} and
\citet{s86}. In a response surface methodology context,  complete information means that the
investigator understands the population in such a way that it is possible to propose a
\textit{value function} reflecting the importance of each response variable. In partial
information,  the investigator knows the main response variable of the study very well and this
is sufficient support for the research. Finally,  under null information,  the researcher only
possesses information about the estimators of the response surface parameter, and with this
elements an appropriate solution can be found too.

In general, an approach for solving a multiobjective mathematical program consist of proposing
an equivalent nonlinear scalar mathematical program, i.e. as a solution of (\ref{equiv1}) is
proposed the following problem
\begin{equation}\label{eq1}
  \begin{array}{c}
    \build{\min}{}{\mathbf{x}}f\left(\widehat{\mathbf{Y}}(\mathbf{x})\right)\\
    \mbox{subject to} \\
    \mathbf{x} \in \mathfrak{X},
  \end{array}
\end{equation}
where $f(\cdot)$ defines a functional ($f(\cdot)$ is a function that takes functions as its
argument, i.e. a function whose domain is a set of functions). Moreover, in the context of
multiobjective mathematical programmming, the functional $f(\cdot)$ is such that if
$\mathfrak{M} \subset \Re^{r}$ denotes a set of multiresponse surface functions, then
\begin{center}
\begin{minipage}[t]{4.5in}
\begin{it}
    The \textbf{functional} is a function $f: \mathfrak{M} \rightarrow \Re$ such
    that $\min\widehat{\mathbf{Y}} (\mathbf{x^{*}}) < \min
    \widehat{\mathbf{Y}} (\mathbf{x}_{1}) \Leftrightarrow
    f(\widehat{\mathbf{Y}} (\mathbf{x^{*}})) <
    f(\widehat{\mathbf{Y}} (\mathbf{x}_{1})),  \quad \mathbf{x}^{*}\neq
    \mathbf{x}_{1}$.
\end{it}
\end{minipage}
\end{center}
In order to consider a greater number of potential solutions of (\ref{equiv1}), usually studied
in the multicriteria mathematical programming, the following alternative problem to (\ref{eq1})
can be proposed
\begin{equation}\label{eq2}
  \begin{array}{c}
    \build{\min}{}{\mathbf{x}}f\left(\widehat{\mathbf{Y}}(\mathbf{x})\right)\\
    \mbox{subject to} \\
    \mathbf{x} \in \mathfrak{X} \cap \mathfrak{S},
  \end{array}
\end{equation}
where $\mathfrak{S}$ is a subset generated by additional potential constraints, generally
derived by a particular technique used for establishing the equivalent scalar mathematical
program (\ref{eq1}). In some particular cases of (\ref{eq1}),  a new fixed parameter may
appear, a vector of response weights $\mathbf{w} = \left(w_{1},  w_{2}, \dots, w_{r}\right)'$,
and/or a vector of target values for the response vector $\boldsymbol{\tau} = \left(\tau_{1},
\tau_{2}, \dots, \tau_{r}\right)'$. Particular examples of this equivalent univariate objective
mathematical programming are the use of goal programming, see \citet{kea:08}, and of the
$\epsilon$-constraint model, see \citet{b:75},  among many others. In particular, under the
$\epsilon$-constraint model, (\ref{eq2}) is proposed as
\begin{equation}\label{eq3}
  \begin{array}{c}
    \build{\min}{}{\mathbf{x}}\widehat{Y}_{j}(\mathbf{x})\\
    \mbox{subject to} \\
    \begin{array}{c}
      \widehat{Y}_{1} (\mathbf{x}) \leq \tau_{1}\\
      \vdots\\
      \widehat{Y}_{j-1} (\mathbf{x}) \leq \tau_{j-1} \\
      \widehat{Y}_{j+1} (\mathbf{x}) \leq \tau_{j+1} \\
      \vdots\\
      \widehat{Y}_{r} (\mathbf{x}) \leq \tau_{r} \\
    \end{array}\\
    \mathbf{x} \in \mathfrak{X}.
  \end{array}
\end{equation}

\section{\normalsize Characterisation of the critical point}\label{sec4}

In the rest of the paper we shall develop the theory of the problem (\ref{eq1}); it is easy to
see  that this problem can be extended with minor modifications to the problem (\ref{eq2}).

Let $\mathbf{x}^{*}(\widehat{\mathbb{B}}) \in \Re^{n}$ be the unique optimal solution of
program (\ref{eq1}) with the corresponding Lagrange multiplier
$\lambda^{*}(\widehat{\mathbb{B}}) \in \Re$. The Lagrangian is defined by
\begin{equation}\label{la}
    L(\mathbf{x}, \lambda; \widehat{\mathbb{B}}) = f\left(\widehat{\mathbf{Y}}(\mathbf{x})\right) + \lambda(||\mathbf{x}||^{2} -
  c^{2}).
\end{equation}
Similarly, $\mathbf{x}^{*}(\mathbb{B}) \in \Re^{n}$ denotes the unique optimal solution of
program (\ref{opp}) with the corresponding Lagrange multiplier $\lambda^{*}(\mathbb{B}) \in
\Re$.

Now we establish the local Kuhn-Tucker conditions that guarantee that the Kuhn-Tucker point
$\mathbf{r}^{*}(\widehat{\mathbb{B}}) = \left[\mathbf{x}^{*}(\widehat{\mathbb{B}}),
\lambda^{*}(\widehat{\mathbb{B}})\right]' \in \Re^{n+1}$ is a unique global minimum of convex
program (\ref{eq1}). First recall that for $f:\Re^{n} \rightarrow \Re$,
$\displaystyle\frac{\partial f}{\partial \mathbf{x}} \equiv \nabla_{\mathbf{x}}$ denotes the
gradient of function $f$.

\begin{theorem}\label{teo1}
The necessary and sufficient conditions that a point $\mathbf{x}^{*}(\widehat{\mathbb{B}}) \in
\Re^{n}$ for arbitrary fixed $\widehat{\mathbb{B}} \in \Re^{p}$, be a unique global minimum of
the convex program (\ref{eq1}) is that, $\mathbf{x}^{*}(\widehat{\mathbb{B}})$ and the
corresponding Lagrange multiplier $\lambda^{*}(\widehat{\mathbb{B}}) \in \Re$, fulfill the
Kuhn-Tucker first order conditions
\begin{eqnarray}\label{kt1}
  \nabla_{\mathbf{x}}L(\mathbf{x}, \lambda; \widehat{\mathbb{B}}) =
   \nabla_{\mathbf{x}}f\left(\widehat{\mathbf{Y}}(\mathbf{x})\right) +
   2\lambda(\widehat{\mathbb{B}}) \mathbf{x} &=& \mathbf{0} \\
  \label{kt2}
  \nabla_{\lambda}L(\mathbf{x}, \lambda; \widehat{\mathbb{B}}) = ||\mathbf{x}||^{2} - c^{2} & \leq & 0 \\
  \label{kt3}
  \lambda(\widehat{\mathbb{B}})(||\mathbf{x}||^{2} - c^{2}) &=& 0\\
  \label{kt4}
  \lambda(\widehat{\mathbb{B}}) &\geq& 0
\end{eqnarray}
\end{theorem}
In addition, assume that strict complementarity slackness holds at $\mathbf{x}^{*}(\mathbb{B})$
with respect to $\lambda^{*}(\mathbb{B})$, that is
\begin{equation}\label{a21}
    \lambda^{*}(\mathbb{B}) > 0 \Leftrightarrow ||\mathbf{x}||^{2} - c^{2} = 0.
\end{equation}
Analogously, the Khun-Tucker condition (\ref{kt1}) to (\ref{kt4}) for $\widehat{\mathbb{B}} =
\mathbb{B}$ are stated next.
\begin{corollary}\label{cor1}
The necessary and sufficient conditions that a point $\mathbf{x}^{*}(\mathbb{B}) \in \Re^{n}$
for arbitrary fixed $\mathbb{B} \in \Re^{p}$, be a unique global minimum of the convex program
(\ref{opp}) is that, $\mathbf{x}^{*}(\mathbb{B})$ and the corresponding Lagrange multiplier
$\lambda^{*}(\mathbb{B}) \in \Re$, fulfill the Kuhn-Tucker first order conditions
\begin{eqnarray}\label{ktp1}
  \nabla_{\mathbf{x}}L(\mathbf{x}, \lambda; \mathbb{B}) =
   \nabla_{\mathbf{x}}f\left(\widehat{\mathbf{Y}}(\mathbf{x})\right) +
   2\lambda(\mathbb{B}) \mathbf{x} &=& \mathbf{0} \\
  \label{ktp2}
  \nabla_{\lambda}L(\mathbf{x}, \lambda; \mathbb{B}) = ||\mathbf{x}||^{2} - c^{2} & \leq & 0 \\
  \label{ktp3}
  \lambda(\mathbb{B})(||\mathbf{x}||^{2} - c^{2}) &=& 0\\
  \label{ktp4}
  \lambda(\mathbb{B}) &\geq& 0
\end{eqnarray}
and $\lambda(\mathbb{B}) = 0$ when $||\mathbf{x}||^{2} - c^{2} < 0$ at
$\left[\mathbf{x}^{*}(\mathbb{B}),\lambda^{*}(\mathbb{B})\right]'$.
\end{corollary}
Observe that, due to the strict convexity of the constraint and objective function, the
second-order sufficient condition is evidently fulfilled for the convex program (\ref{eq1}).

The next result states the existence of a once continuously differentiable solution to program
(\ref{eq1}), see \citet{fg:82}.

\begin{theorem}\label{teo2}
Assume that (\ref{a21}) holds  and the second-order sufficient condition is satisfied by the
convex program (\ref{eq1}). Then
\begin{enumerate}
  \item $\mathbf{x}^{*}(\mathbb{B})$ is a unique global minimum of program (\ref{opp})
  and $\lambda^{*}(\mathbb{B})$ is also unique.
  \item For $\widehat{\mathbb{B}} \in V_{\varepsilon}(\mathbb{B})$
  (is an $\varepsilon-$neighborhood or open ball), there exist a unique once continuously
  differentiable vector function
  $$
    \mathbf{r}^{*}(\widehat{\mathbb{B}}) =
    \left[
    \begin{array}{c}
      \mathbf{x}^{*}(\widehat{\mathbb{B}}) \\
      \lambda^{*}(\widehat{\mathbb{B}})
    \end{array}
    \right] \in \Re^{n+1}
  $$
  satisfying the second order sufficient conditions of problem (\ref{opp}), such that $\mathbf{r}^{*}(\mathbb{B}) =
\left[\mathbf{x}^{*}(\mathbb{B}), \lambda^{*}(\mathbb{B})\right]'$ and hence,
$\mathbf{x}^{*}(\widehat{\mathbb{B}})$ is a unique global minimum of problem (\ref{eq1}) with
associated unique Lagrange multiplier $\lambda^{*}(\widehat{\mathbb{B}})$.
  \item For $\widehat{\mathbb{B}} \in V_{\varepsilon}(\mathbb{B})$, the status of the constraint is unchanged and
  $\lambda^{*}(\widehat{\mathbb{B}}) > 0 \Leftrightarrow ||\mathbf{x}||^{2} - c^{2} = 0$ holds.
\end{enumerate}
\end{theorem}

\section{\normalsize Asymptotic normality of the critical point}\label{sec5}

This section considers  the statistical and mathematical programming aspects of the sensitivity
analysis of the optimum of a estimated multiresponse surface model.

\begin{theorem} \label{teo3}Assume:
\begin{enumerate}
   \item For any $\widehat{\mathbb{B}} \in V_{\varepsilon}(\mathbb{B})$,
   the second-order sufficient condition is fulfilled for the convex program (\ref{eq1})
   such that the second order derivatives
   $$
     \frac{\partial^{2} L(\mathbf{x}, \lambda; \widehat{\mathbb{B}})}{\partial \mathbf{x}\partial
     \mathbf{x}'}, \  \frac{\partial^{2} L(\mathbf{x}, \lambda; \widehat{\mathbb{B}})}{\partial \mathbf{x}\partial
     \Vec'\widehat{\mathbb{B}}}, \  \frac{\partial^{2} L(\mathbf{x}, \lambda; \widehat{\mathbb{B}})}{\partial \mathbf{x}\partial
     \lambda}, \  \frac{\partial^{2} L(\mathbf{x}, \lambda; \widehat{\mathbb{B}})}{\partial \lambda \partial
     \mathbf{x}'}, \  \frac{\partial^{2} L(\mathbf{x}, \lambda; \widehat{\mathbb{B}})}{\partial \lambda
     \partial \Vec\widehat{\mathbb{B}}}
   $$
   exist and are continuous in $\left[\mathbf{x}^{*}(\widehat{\mathbb{B}}),
   \lambda^{*}(\widehat{\mathbb{B}})\right]' \in V_{\varepsilon}(\left[\mathbf{x}^{*}(\mathbb{B}),
   \lambda^{*}(\mathbb{B})\right]')$ and
   $$
     \frac{\partial^{2} L(\mathbf{x}, \lambda; \widehat{\mathbb{B}})}{\partial \mathbf{x}\partial
     \mathbf{x}'}
   $$
   is positive definite.
  \item $\widehat{\mathbb{B}}_{\nu}$, the estimator of the true parameter vector
  $\mathbb{B}_{\nu}$, is based on a sample of size $N_{\nu}$ such that
  $$
    \sqrt{N_{\nu}}(\widehat{\mathbb{B}}_{\nu} - \mathbb{B}_{\nu}) \sim \mathcal{N}_{p \times r}
    (\mathbb{B},  \mathbf{\Theta}), \quad
    \frac{1}{N_{\nu}}\boldsymbol{\Theta} =  (\mathbf{X}'\mathbf{X})^{-1} \otimes \mathbf{\Sigma}.
  $$
  \item (\ref{a21}) holds for $\widehat{\mathbb{B}} = \mathbb{B}$. Then
  asymptotically
  $$
    \sqrt{N_{\nu}}\left[\mathbf{x}^{*}(\widehat{\mathbb{B}}) -
     \mathbf{x}^{*}(\mathbb{B})\right] \build{\rightarrow}{d}{} \mathcal{N}_{n}(\mathbf{0}_{n}, \boldsymbol{\Xi}),
  $$
  where the $n \times n$ variance-covariance matrix
  $$
    \boldsymbol{\Xi} = \left(\frac{\partial \mathbf{x}^{*}(\widehat{\mathbb{B}})}{\partial
    \Vec\widehat{\mathbb{B}}}\right)\widehat{\boldsymbol{\Theta}} \left(\frac{\partial \mathbf{x}^{*}
    (\widehat{\mathbb{B}})}{\partial \Vec\widehat{\mathbb{B}}}\right)', \quad
    \frac{1}{N_{\nu}}\widehat{\boldsymbol{\Theta}} =  (\mathbf{X}'\mathbf{X})^{-1} \otimes \widehat{\mathbf{\Sigma}}
  $$
  such that all elements of $\left(\partial \mathbf{x}^{*}(\widehat{\mathbb{B}})/\partial
  \Vec\widehat{\mathbb{B}}\right)$ are continuous on any $\widehat{\mathbb{B}}
  \in V_{\varepsilon}(\mathbb{B})$; furthermore
  $$
    \left(\frac{\partial \mathbf{x}^{*}(\widehat{\mathbb{B}})}{\partial
    \Vec\widehat{\mathbb{B}}}\right) = \left[\mathbf{I} - \mathbf{P}^{-1}\mathbf{Q}(\mathbf{Q}'\mathbf{P}^{-1}
    \mathbf{Q})^{-1}\mathbf{Q}'\right]\mathbf{P}^{-1} \mathbf{G},
  $$
  where
  $$
    \mathbf{P} = \frac{\partial^{2} L(\mathbf{x}, \lambda; \widehat{\mathbb{B}})}{\partial \mathbf{x}\partial \mathbf{x}'}
  $$
  $$
    \mathbf{Q} = \frac{\partial^{2} L(\mathbf{x}, \lambda; \widehat{\mathbb{B}})}{\partial \lambda \partial \mathbf{x}}
  $$

  $$
    \mathbf{G} = \frac{\partial^{2} L(\mathbf{x}, \lambda; \widehat{\mathbb{B}})}{\partial \mathbf{x}\partial \Vec'\widehat{\mathbb{B}}}
  $$
\end{enumerate}
\end{theorem}
\begin{proof}
According to Theorem \ref{teo1} and Corollary \ref{cor1}, the Kuhn-Tucker conditions
(\ref{kt1})--(\ref{kt4}) at
$\left[\mathbf{x}^{*}(\widehat{\mathbb{B}}),\lambda^{*}(\widehat{\mathbb{B}})\right]'$
and the conditions (\ref{ktp1})--(\ref{ktp4})
at
$\left[\mathbf{x}^{*}(\mathbb{B}),\lambda^{*}(\mathbb{B})\right]'$ are fulfilled for
mathematical programs (\ref{opp}) and (\ref{eq1}), respectively. From conditions
(\ref{ktp1})--(\ref{ktp4}) of Corollary \ref{cor1}, the following system equation
\begin{eqnarray}\label{ktpp1}
  \nabla_{\mathbf{x}}L(\mathbf{x}, \lambda; \mathbb{B}) =
   \nabla_{\mathbf{x}}f\left(\widehat{\mathbf{Y}}(\mathbf{x})\right) +
   2\lambda(\mathbb{B}) \mathbf{x} &=& \mathbf{0}\\
  \label{ktpp2}
  \nabla_{\lambda}L(\mathbf{x}, \lambda; \mathbb{B}) = ||\mathbf{x}||^{2} - c^{2} & = & 0,
\end{eqnarray}
has a solution $\mathbf{x}^{*}(\mathbb{B}), \lambda^{*}(\mathbb{B}) > 0$,  $\mathbb{B}$.

The nonsingular Jacobian matrix of the continuously differentiable functions (\ref{ktpp1}) and
(\ref{ktpp2}) with respect to $\mathbf{x}$ and $\lambda$ at
$\left[\mathbf{x}^{*}(\widehat{\mathbb{B}}),\lambda^{*}(\widehat{\mathbb{B}})\right]'$ is
$$
  \left(
     \begin{array}{cc}
       \displaystyle\frac{\partial^{2} L(\mathbf{x}, \lambda; \widehat{\mathbb{B}})}{\partial \mathbf{x}\partial
       \mathbf{x}'} & \displaystyle\frac{\partial^{2} L(\mathbf{x}, \lambda; \widehat{\mathbb{B}})}{\partial \lambda \partial
       \mathbf{x}} \\
       \displaystyle\frac{\partial^{2} L(\mathbf{x}, \lambda; \widehat{\mathbb{B}})}{\partial \mathbf{x}'\partial
       \lambda} & 0
     \end{array}
     \right ).
$$
According to the implicit functions theorem, there is a neighborhood
$V_{\varepsilon}(\mathbb{B})$ such that for arbitrary $\widehat{\mathbb{B}} \in
V_{\varepsilon}(\mathbb{B})$, the system (\ref{ktpp1}) and (\ref{ktpp2}) has a unique solution
$\mathbf{x}^{*}(\widehat{\mathbb{B}})$, $\lambda^{*}(\widehat{\mathbb{B}})$,
$\widehat{\mathbb{B}}$ and by Theorem \ref{teo2}, the components of
$\mathbf{x}^{*}(\widehat{\mathbb{B}})$, $\lambda^{*}(\widehat{\mathbb{B}})$ are continuously
differentiable function of $\widehat{\mathbb{B}}$, see \citet{bsh:74}. Their derivatives are
given by
\begin{eqnarray}\label{xx}
  \left (
     \begin{array}{c}
       \displaystyle\frac{\partial \mathbf{x}^{*}(\widehat{\mathbb{B}})}{\partial
       \Vec\widehat{\mathbb{B}}} \\[2ex]
       \displaystyle\frac{\partial \lambda^{*}(\widehat{\mathbb{B}})}{\partial
       \Vec\widehat{\mathbb{B}}}
     \end{array}
  \right )
  &=& - \left(
     \begin{array}{cc}
       \displaystyle\frac{\partial^{2} L(\mathbf{x}, \lambda; \widehat{\mathbb{B}})}{\partial \mathbf{x}\partial
       \mathbf{x}'} & \displaystyle\frac{\partial^{2} L(\mathbf{x}, \lambda; \widehat{\mathbb{B}})}{\partial \lambda \partial
       \mathbf{x}} \\
       \displaystyle\frac{\partial^{2} L(\mathbf{x}, \lambda; \widehat{\mathbb{B}})}{\partial \mathbf{x}'\partial
       \lambda} & 0
     \end{array}
     \right )^{-1}
     \left(
        \begin{array}{c}
          \displaystyle\frac{\partial^{2} L(\mathbf{x}, \lambda; \widehat{\mathbb{B}})}{\partial \mathbf{x}\partial
          \Vec'\widehat{\mathbb{B}}} \\
          0
        \end{array}
     \right).
\end{eqnarray}
The explicit form of $(\partial \mathbf{x}^{*}(\widehat{\mathbb{B}})/ \partial
\Vec\widehat{\mathbb{B}})$ follows from (\ref{xx}) and by the formula
$$
  \left(
    \begin{array}{cc}
      \mathbf{P} & \mathbf{Q} \\
      \mathbf{Q}' & \mathbf{0}
    \end{array}
  \right )^{-1}
  =
  \left(
    \begin{array}{cc}
      [\mathbf{I} - \mathbf{P}^{-1}\mathbf{Q}(\mathbf{Q}'\mathbf{P}^{-1}\mathbf{Q})^{-1}\mathbf{Q}']\mathbf{P}^{-1}
      & \mathbf{P}^{-1}\mathbf{Q}(\mathbf{Q}'\mathbf{P}^{-1}\mathbf{Q})^{-1} \\
      (\mathbf{Q}'\mathbf{P}^{-1}\mathbf{Q})^{-1}\mathbf{Q}'\mathbf{P}^{-1} & -(\mathbf{Q}'\mathbf{P}^{-1}\mathbf{Q})^{-1}
    \end{array}
  \right ),
$$
where $\mathbf{P}$ is symmetric and  nonsingular.

Then from assumption 2, \citet[(iii), p. 388]{r:73} and \citet[Theorem 14.6-2, p.
493]{betal:91} (see also \citet[p. 353]{cr:46}) we have
\begin{equation}\label{AN}
    \sqrt{N_{\nu}}\left[\mathbf{x}^{*}(\widehat{\mathbb{B}}) -
     \mathbf{x}^{*}(\mathbb{B})\right] \build{\rightarrow}{d}{} \mathcal{N}_{n}\left(\mathbf{0}_{n},
     \left(\frac{\partial \mathbf{x}^{*}(\mathbb{B})}{\partial
    \Vec\widehat{\mathbb{B}}}\right)\boldsymbol{\Theta} \left(\frac{\partial \mathbf{x}^{*}
    (\mathbb{B})}{\partial \Vec\widehat{\mathbb{B}}}\right)'\right).
\end{equation}
Finally note that all elements of $(\partial \mathbf{x}^{*}/ \partial \widehat{\mathbb{B}})$
are continuous on $V_{\varepsilon}(\mathbb{B})$, so that the asymptotical distribution
(\ref{AN}) can be substituted by
$$
  \sqrt{N_{\nu}}\left[\mathbf{x}^{*}(\widehat{\mathbb{B}}) -
     \mathbf{x}^{*}(\mathbb{B})\right] \build{\rightarrow}{d}{} \mathcal{N}_{n}\left(\mathbf{0}_{n},
     \left(\frac{\partial \mathbf{x}^{*}(\widehat{\mathbb{B}})}{\partial\Vec
    \widehat{\mathbb{B}}}\right)\widehat{\boldsymbol{\Theta}} \left(\frac{\partial \mathbf{x}^{*}
    (\widehat{\mathbb{B}})}{\partial \Vec\widehat{\mathbb{B}}}\right)'\right),
$$
see \citet[(iv), pp.388--389]{r:73}.
\end{proof}

As a particular case, assume that the functional in (\ref{eq1}) is defined as
$$
  f\left(\widehat{\mathbf{Y}}(\mathbf{x})\right) = \sum_{k = 1}^{r} w_{k}
  \widehat{Y}_{k}(\mathbf{x}), \quad \sum_{k = 1}^{r} w_{k} = 1,
$$
with $w_{k}$ known constants. Then,

\begin{corollary} Suppose:
\begin{enumerate}
   \item For any $\widehat{\mathbb{B}} \in V_{\varepsilon}(\mathbb{B})$,
   the second-order sufficient condition is fulfilled for the convex program (\ref{eq1})
   such that the second order derivatives
   $$
     \frac{\partial^{2} L(\mathbf{x}, \lambda; \widehat{\mathbb{B}})}{\partial \mathbf{x}\partial
     \mathbf{x}'}, \  \frac{\partial^{2} L(\mathbf{x}, \lambda; \widehat{\mathbb{B}})}{\partial \mathbf{x}\partial
     \Vec'\widehat{\mathbb{B}}}, \  \frac{\partial^{2} L(\mathbf{x}, \lambda; \widehat{\mathbb{B}})}{\partial \mathbf{x}\partial
     \lambda}, \  \frac{\partial^{2} L(\mathbf{x}, \lambda; \widehat{\mathbb{B}})}{\partial \lambda \partial
     \mathbf{x}'}, \  \frac{\partial^{2} L(\mathbf{x}, \lambda; \widehat{\mathbb{B}})}{\partial \lambda
     \partial \Vec\widehat{\mathbb{B}}}
   $$
   exist and are continuous in $\left[\mathbf{x}^{*}(\widehat{\mathbb{B}}),
   \lambda^{*}(\widehat{\mathbb{B}})\right]' \in V_{\varepsilon}(\left[\mathbf{x}^{*}(\mathbb{B}),
   \lambda^{*}(\mathbb{B})\right]')$ and
   $$
     \frac{\partial^{2} L(\mathbf{x}, \lambda; \widehat{\mathbb{B}})}{\partial \mathbf{x}\partial
     \mathbf{x}'}
   $$
   is positive definite.
  \item $\widehat{\mathbb{B}}_{\nu}$, the estimator of the true parameter vector
  $\mathbb{B}_{\nu}$,  is based on a sample of size $N_{\nu}$  such that
  $$
    \sqrt{N_{\nu}}(\widehat{\mathbb{B}}_{\nu} - \mathbb{B}_{\nu}) \sim \mathcal{N}_{p \times r}
    (\mathbb{B},  \mathbf{\Theta}), \quad
    \frac{1}{N_{\nu}}\boldsymbol{\Theta} =  (\mathbf{X}'\mathbf{X})^{-1} \otimes \mathbf{\Sigma}.
  $$
  \item (\ref{a21}) holds for $\widehat{\mathbb{B}} = \mathbb{B}$. Then
  asymptotically
  $$
    \sqrt{N_{\nu}}\left[\mathbf{x}^{*}(\widehat{\mathbb{B}}) -
     \mathbf{x}^{*}(\mathbb{B})\right] \build{\rightarrow}{d}{} \mathcal{N}_{n}(\mathbf{0}_{n}, \boldsymbol{\Xi})
  $$
  where the $n \times n$ variance-covariance matrix
  $$
    \boldsymbol{\Xi} = \left(\frac{\partial \mathbf{x}^{*}(\widehat{\mathbb{B}})}{\partial
    \Vec\widehat{\mathbb{B}}}\right)\widehat{\boldsymbol{\Theta}} \left(\frac{\partial \mathbf{x}^{*}
    (\widehat{\mathbb{B}})}{\partial \Vec\widehat{\mathbb{B}}}\right)', \quad
    \frac{1}{N_{\nu}}\widehat{\boldsymbol{\Theta}} =  (\mathbf{X}'\mathbf{X})^{-1} \otimes \widehat{\mathbf{\Sigma}}
  $$
  such that all elements of $\left(\partial \mathbf{x}^{*}(\widehat{\mathbb{B}})/\partial
  \Vec\widehat{\mathbb{B}}\right)$ are continuous on any $\widehat{\mathbb{B}}
  \in V_{\varepsilon}(\mathbb{B})$; furthermore
  $$
    \left(\frac{\partial \mathbf{x}^{*}(\widehat{\mathbb{B}})}{\partial
    \Vec\widehat{\mathbb{B}}}\right) = \mathbf{S}^{-1}
    \left(\frac{\mathbf{x}^{*}(\widehat{\mathbb{B}})\mathbf{x}^{*}(\widehat{\mathbb{B}})'
    \mathbf{S}^{-1}}{\mathbf{x}^{*}(\widehat{\mathbb{B}})'
    \mathbf{S}^{-1}\mathbf{x}^{*}(\widehat{\mathbb{B}})}
    - \mathbf{I}_{n}\right)\mathbf{M}\left(\mathbf{x}^{*}(\widehat{\mathbb{B}})\right),
  $$
  where
  $$
    \mathbf{S} = \frac{\partial^{2} L(\mathbf{x}, \lambda; \widehat{\mathbb{B}})}{\partial \mathbf{x}\partial
     \mathbf{x}'}= 2\sum_{k = 1}^{r}w_{k} \widehat{\mathbf{B}}_{k}-2\lambda^{*}
    (\widehat{\mathbb{B}})\mathbf{I}_{n}.
  $$
  and
\begin{eqnarray*}
  \mathbf{M}(\mathbf{x}) &=& \nabla_{\mathbf{x}}\mathbf{z}'(\mathbf{x}) = \frac{\partial \mathbf{z}'(\mathbf{x})}{\partial \mathbf{x}} \\
     &=& (\mathbf{0}\vdots \mathbf{I}_{n} \vdots 2\diag(\mathbf{x}) \vdots \mathbf{C}_{1}\vdots\cdots
     \vdots\mathbf{C}_{n-1}) \in \Re^{n \times p},
\end{eqnarray*}
with
$$
  \mathbf{C}_{i} =
  \left(
     \begin{array}{c}
       \mathbf{0}'_{1} \\
       \vdots \\
       \mathbf{0}'_{i-1} \\
       \mathbf{x}'\mathbf{A}_{i} \\
       x_{i}\mathbf{I}_{n-i}
     \end{array}
  \right), i = 1, \dots, n-1, \quad \mathbf{0}_{j} \in \Re^{n-i}, j = 1, \dots, i-1;
$$
observing that when $i=1$ (i.e. j = 0), this row does not appear in $\mathbf{C}_{1}$; and
$$
  \mathbf{A}_{i} =
  \left (
     \begin{array}{c}
       \mathbf{0}'_{1} \\
       \vdots \\
       \mathbf{0}'_{i} \\
       \mathbf{I}_{n-i}
     \end{array}
  \right), \quad \mathbf{0}'_{k} \in \Re^{n-i}, k = 1, \dots,i.
$$
\end{enumerate}
\end{corollary}
\begin{proof} The required result  follows from Theorem \ref{teo3} and observing that in this
particular case
\begin{eqnarray*}
  \nabla_{\mathbf{x}}L(\mathbf{x}, \lambda; \mathbb{B}) =
  \left\{
    \begin{array}{l}
      \mathbf{M}(\mathbf{x}) \displaystyle\sum_{k = 1}^{r}w_{k} \boldsymbol{\beta}_{k} +
      2\lambda(\mathbb{B}) \mathbf{x}\\ \quad\mbox{or}\\
      \displaystyle\sum_{k = 1}^{r}w_{k}\left[\boldsymbol{\beta}_{1k} +
      2\mathbf{B}_{k}\mathbf{x}\right] + 2\lambda(\mathbb{B})\mathbf{x}
    \end{array}
  \right \}
   &=& \mathbf{0}\\
  \label{ktpp2}
  \nabla_{\lambda}L(\mathbf{x}, \lambda; \boldsymbol{\beta}) = ||\mathbf{x}||^{2} - c^{2} & = & 0
\end{eqnarray*}
\end{proof}

\section*{\normalsize Conclusions}

As a consequence of Theorem \ref{teo2} now is feasible to establish confidence  regions and
intervals and hypothesis tests on the critical point, see \citet[Section 14.6.4, pp.
498--500]{betal:91}; it is also possible to identify operating conditions as regions
or intervals instead of isolated points.

The results of this paper can be taken as a good first approximation to the exact problem.
However, unfortunately in many applications the number of observations is relatively small and
perhaps the results obtained in this work should be applied with caution.

\section*{\normalsize Acknowledgments}

%The authors also wish to thank the referees and Editor who handled this article for their very
%helpful suggestions.
This paper was written during J. A. D\'{\i}az-Garc\'{\i}a's stay as a professor at the
Department of Statistics and O. R of the University of Granada, Espa\~{n}a. F. Caro was
supported by the project No. 158 of University of Medellin.

\end{document}